\documentclass[11pt]{article}

\usepackage{amsmath, amsthm, amsfonts, enumerate, mathrsfs, pmat}
\usepackage[pdftex]{graphicx}

\newtheorem{thm}{Theorem}[section]

\newtheorem{lem}[thm]{Lemma}
\newtheorem{prop}[thm]{Proposition}

\theoremstyle{definition}

\theoremstyle{remark}
\newtheorem*{rem}{Remark} 
\theoremstyle{definition}
\newtheorem{ex}[thm]{Example}

\title{Constructing New Realisable Lists from Old in the NIEP}
\date{}
\author{
	Richard Ellard$^1$\thanks{The authors' work was supported by Science Foundation Ireland under Grant 11/RFP.1/MTH/3157.}\ \ and Helena \v{S}migoc$^{2*}$ \\
	{\small School of Mathematical Sciences,} \\ 
	{\small University College Dublin,} \\ 
	{\small Belfield, Dublin 4, Ireland}\\
	{\small $^1$email: richardellard@gmail.com}\\
	{\small $^2$email: helena.smigoc@ucd.ie}\\
}

\begin{document}
\allowdisplaybreaks

\maketitle

\begin{abstract}

Given a list of complex numbers $\sigma:=(\lambda_1,\lambda_2,\ldots,\lambda_m)$, we say that
$\sigma$ is \emph{realisable} if $\sigma$ is the spectrum of some (entrywise) nonnegative matrix.
The Nonnegative Inverse Eigenvalue Problem (or NIEP) is the problem of categorising all realisable
lists.

Given a realisable list $(\rho,\lambda_2,\lambda_3,\ldots,\lambda_m)$, where $\rho$
is the Perron eigenvalue and $\lambda_2$ is real, we find families of lists
\[
  (\mu_1,\mu_2,\ldots,\mu_n),
\]
for which
\[
	(\mu_1,\mu_2,\ldots,\mu_n,\lambda_3,\lambda_4,\ldots,\lambda_m)
\]
is realisable. In addition, given a realisable list
\[(\rho,\alpha+i\beta,\alpha-i\beta,\lambda_4,\lambda_5,\ldots,\lambda_m),\] where $\rho$ is the Perron eigenvalue and $\alpha$ and $\beta$ are real, we find families of lists
$
  (\mu_1,\mu_2,\mu_3,\mu_4),
$
for which
\[
	(\mu_1,\mu_2,\mu_3,\mu_4,\lambda_4,\lambda_5,\ldots,\lambda_m)
\]
is realisable.

{\bf AMS classification:} 15A18, 15A29

{\bf Keywords:} Nonnegative matrices, Nonnegative Inverse Eigenvalue Problem, Companion matrix

\end{abstract}

\section{Introduction}\label{sec:Introduction}

We denote the spectrum of a matrix $A$ by $\sigma(A)$. We say that $A$ is nonnegative if it is
entrywise nonnegative and in this case we write $A\geq0$. In general, if $A,B\in\mathbb{R}^{n\times n}$ or $y,z\in\mathbb{R}^n$, we will use notation such as $A\geq B$ or $y\geq z$ if the inequalities hold entrywise. For a list of complex numbers $\sigma:=(\lambda_1,\lambda_2,\ldots,\lambda_n)$, we define $s_m(\sigma):=\sum_{i=1}^n\lambda_i^m$. $I_n$ denotes the $n\times n$ identity matrix.

We call $\sigma$ \emph{realisable} if there exists a nonnegative matrix $A$ with spectrum $\sigma$ and in this case, we say that $A$ \emph{realises} $\sigma$. The Nonnegative Inverse Eigenvalue Problem (or NIEP) is the problem of categorising all realisable lists.

We begin by stating some well-known necessary conditions for a list to be realisable. Let $\sigma:=(\lambda_1,\lambda_2,\ldots,\lambda_n)$ be the spectrum of a  nonnegative
matrix $A$. Then
	\begin{enumerate}[(i)]
		\item[\textup{(i)}]
		$\sigma$ is closed under complex conjugation, i.e. 
		$		
\overline{\sigma}:=\left(\overline{\lambda_1},\overline{\lambda_2},\ldots,\overline{\lambda_n}\right)=\sigma;
		$
		\item[\textup{(ii)}]
		$
			\max_i|\lambda_i|\in\sigma;
		$
		\item[\textup{(iii)}]
		$
			s_m(\sigma)\geq0
		$
		for every positive integer $m$;
		\item[\textup{(iv)}]
		$
			s_k(\sigma)^m\leq n^{m-1}s_{km}(\sigma)
		$
		for all positive integers $k$ and $m$.
	\end{enumerate}
Condition (i) follows from the fact that the characteristic polynomial of $A$ has real coefficients.
Condition (ii) says that the spectral radius of $A$, $\rho$ say, is an eigenvalue of $A$. This result forms part of the well-known Perron-Frobenius theory of nonnegative matrices. The eigenvalue $\rho$ is known as the \emph{Perron eigenvalue} of $A$ and the corresponding eigenvector is known as the \emph{Perron eigenvector}. We will always write the Perron eigenvalue as the first entry in a realisable list. Condition (iii) follows from the fact that $s_m(\sigma)$ is the trace of $A^m$. The inequalities in (iv) are called the \emph{JLL conditions}. They were proved by Loewy and London \cite{LoewyLondon} and independently by Johnson \cite{Johnson}.

We denote by $e$ the vector of appropriate size with every entry equal to 1, i.e. 
$e:=[\begin{array}{cccc}1 & 1 & \cdots & 1\end{array}]^T$. The following useful result---due to Johnson \cite{Johnson}---allows us to assume without loss of generality that the Perron eigenvector of a realising matrix is $e$. A proof can also be found in \cite{Guo}.

\begin{lem}\label{lem:e}\textup{\bf\cite{Johnson}}
	Let $A$ be a nonnegative matrix with Perron eigenvalue $\rho$. Then there exists a
nonnegative matrix $B$, cospectral with $A$, satisfying $Be=\rho e$.
\end{lem}

In the case where all eigenvalues but the Perron have nonpositive real parts, the NIEP has been completely solved by Laffey and \v{S}migoc \cite{LaffeySmigoc}:

\begin{thm}\textup{\bf\cite{LaffeySmigoc}}\label{thm:LS}
  Let $\rho\geq0$ and let $\lambda_2,\lambda_3,\ldots,\lambda_n$ be complex numbers such that
$\mathrm{Re}\,\lambda_i\leq0$ for all $i=2,3,\ldots,n$. Then the list
$\sigma=(\rho,\lambda_2,\lambda_3,\ldots,\lambda_n)$ is the spectrum of a nonnegative
matrix if and only if the following conditions are satisfied:
  \begin{enumerate}[(i)]
    \item[\textup{(i)}]
    $\sigma$ is closed under complex conjugation;
    \item[\textup{(ii)}]
    $s_1(\sigma)\geq0$;
    \item[\textup{(iii)}]
    $s_1(\sigma)^2\leq ns_2(\sigma)$.
  \end{enumerate}
  Furthermore, when the above conditions hold, $\sigma$ may be realised by a matrix of the form
$G+\gamma I_n$, where $G$ is a nonnegative companion matrix with trace zero and $\gamma$ is a
nonnegative scalar.
\end{thm}

\begin{rem}
	The condition that $\mathrm{Re}\,\lambda_i\leq0$ for all $i=2,3,\ldots,n$ in Theorem \ref{thm:LS} can be relaxed to $\mathrm{Re}\,\lambda_i\leq s_1(\sigma)/n$. To see this, note that the quantity
	\[
		ns_2(\sigma)-s_1(\sigma)^2
	\]
	is unchanged by subtracting a scalar from $\sigma$, i.e.
	\begin{multline*}
		ns_2(\rho-\delta,\lambda_2-\delta,\ldots,\lambda_n-\delta)-s_1(\rho-\delta,\lambda_2-\delta,\ldots,\lambda_n-\delta)^2
		\\=
		ns_2(\rho,\lambda_2,\ldots,\lambda_n)-s_1(\rho,\lambda_2,\ldots,\lambda_n)^2
	\end{multline*}
	for all $\delta\in\mathbb{C}$ and hence if $(\rho,\lambda_2,\ldots,\lambda_n)$ satisfies (i)--(iii), then so does $(\rho-s_1(\sigma)/n,\lambda_2-s_1(\sigma)/n,\ldots,\lambda_n-s_1(\sigma)/n)$.
\end{rem}

The results in this paper fall into the category of constructing new realisable lists from known realisable lists. We give some earlier results of this type below. Guo \cite{Guo} gave the following theorem regarding the perturbation of a realisable list:

\begin{thm}\textup{\bf\cite{Guo}}\label{thm:Guo}
	If $(\rho,\lambda_2,\lambda_3,\ldots,\lambda_n)$ is realisable, where $\rho$ is the Perron
eigenvalue and $\lambda_2$ is real, then
	\[
		(\rho+\delta,\lambda_2\pm \delta,\lambda_3,\lambda_4,\ldots,\lambda_n)
	\]
	is realisable for all $\delta\geq0$.
\end{thm}

To generalise Theorem \ref{thm:Guo} to the perturbation of non-real eigenvalues, we have the
following theorem. Result (\ref{eq:GuoGuoDecrease}) is due to Laffey \cite{Laffey} and an alternative proof can be found in \cite{GuoGuo}. Result (\ref{eq:GuoGuoIncrease}) is due to Guo and Guo \cite{GuoGuo}.

\begin{thm}\label{thm:GuoGuo}
	If $(\rho,\alpha+i\beta,\alpha-i\beta,\lambda_4,\lambda_5,\ldots,\lambda_n)$ is realisable, where
$\rho$ is the Perron eigenvalue and $\alpha$ and $\beta$ are real, then for all $\delta\geq0$, the
lists
	\begin{equation}\label{eq:GuoGuoDecrease}
		(\rho+2\delta,\alpha-\delta+i\beta,\alpha-\delta-i\beta,\lambda_4,\lambda_5,\ldots,\lambda_n)
	\end{equation}
	and
	\begin{equation}\label{eq:GuoGuoIncrease}
		(\rho+4\delta,\alpha+\delta+i\beta,\alpha+\delta-i\beta,\lambda_4,\lambda_5,\ldots,\lambda_n)
	\end{equation}
	are realisable.
\end{thm}

\v{S}migoc \cite{SmigocDiagonalElement} gives a different kind of perturbation, in which the
Perron eigenvalue of a realisable list may be replaced by a new list:

\begin{thm}\textup{\bf\cite{SmigocDiagonalElement}}\label{thm:SmigocDiagonalElement}
	Let $(\rho,\lambda_2,\lambda_3,\ldots,\lambda_m)$ be realisable, where $\rho$ is the Perron
eigenvalue and let $(\mu_1,\mu_2,\ldots,\mu_n)$ be the spectrum of a nonnegative matrix with a
diagonal element greater than or equal to $\rho$. Then
	\[
		(\mu_1,\mu_2,\ldots,\mu_n,\lambda_2,\lambda_3,\ldots,\lambda_m)
	\]
	is realisable.
\end{thm}

In \cite{SmigocSubmatrixConstruction}, \v{S}migoc gives a construction to replace both the Perron
eigenvalue and another real eigenvalue:

\begin{thm}\textup{\bf\cite{SmigocSubmatrixConstruction}}\label{thm:SmigocSubmatrixConstruction}
	Let $(\rho,\lambda_2,\lambda_3,\ldots,\lambda_m)$ be realisable, where $\rho$ is
the Perron eigenvalue and $\lambda_2$ is real. Let $a$ and $t_1$ be any nonnegative numbers and let $t_2$ be any real number such that $|t_2|\leq t_1$. Then
	\[
		(\mu_1,\mu_2,\mu_3,\lambda_3,\lambda_4,\ldots,\lambda_m)
	\]
	is realisable, where $\mu_1,\mu_2,\mu_3$ are the roots of the polynomial
	\[
		w(x)=(x-\rho)(x-\lambda_2)(x-a)-(t_1+t_2)x+t_1\lambda_2+t_2\rho.
	\]
\end{thm}

In Section \ref{sec:p=2}, we expand on the work done in \cite{SmigocSubmatrixConstruction} by
presenting some new lists which may replace the eigenvalues $\rho$ and $\lambda_2$. In Section
\ref{sec:p=3}, we give a construction which allows us to replace the Perron eigenvalue and a complex
conjugate pair of eigenvalues, i.e. given a realisable list
\[
	(\rho,\alpha+i\beta,\alpha-i\beta,\lambda_4,\lambda_5,\ldots,\lambda_m),
\]	
where $\rho$ is the Perron eigenvalue and $\alpha$ and $\beta$ are real, we find some conditions on the list $(\mu_1,\mu_2,\mu_3,\mu_4)$ which imply that
\[
	(\mu_1,\mu_2,\mu_3,\mu_4,\lambda_4,\lambda_5,\ldots,\lambda_m)
\]
is realisable.

To this end, we begin by giving a Lemma from \cite{SmigocSubmatrixConstruction}, which is the
foundation of this work:

\begin{lem}\textup{\bf\cite{SmigocSubmatrixConstruction}}\label{lem:SubmatrixConstruction}
	Let the following assumptions hold:
	\begin{enumerate}[(i)]
		\item[\textup{(i)}]
		$Y$ is an invertible matrix with a partition
		$
			Y=\left[
				\begin{array}{cc}
					Y_1 & Y_2
				\end{array}
			\right],
		$
		where $Y_1$ is an $m\times p$ matrix and $Y_2$ is an $m\times m_1$ matrix with
$p+m_1=m$;
		\item[\textup{(ii)}]
		$B$ is an $m\times m$ matrix such that
		\[
			Y^{-1}BY=\left[
				\begin{array}{cc}
					C & E \\
					0 & F
				\end{array}
			\right]
		\]
		for a $p\times p$ matrix $C$ and an $m_1\times m_1$ matrix $F$;
		\item[\textup{(iii)}]
		$M$ is an $n\times n$ matrix with a principal submatrix $C$, partitioned in the
following way:
		\[
			M=\left[
				\begin{array}{cc}
					A & K \\
					L & C
				\end{array}
			\right],
		\]
		where $A$ is an $n_1\times n_1$ matrix and $p+n_1=n$;
		\item[\textup{(iv)}]
		$K=HY_1$ for an $n_1\times m$ matrix $H$.
	\end{enumerate}
	Then for matrices
	\[
		N=\left[
			\begin{array}{cc}
				A & H \\
				Y_1L & B
			\end{array}
		\right]
		\;\; \text{and} \;\;
		Z=\left[
			\begin{array}{cc}
				I_{n_1} & 0 \\
				0 & Y
			\end{array}
		\right],
	\]
	we have
	\[
		Z^{-1}NZ=\left[
			\begin{array}{ccc}
				A & K & HY_2 \\
				L & C & E \\
				0 & 0 & F
			\end{array}
		\right].
	\]
\end{lem}

In particular, Lemma \ref{lem:SubmatrixConstruction} produces a matrix $N$ with spectrum
$\sigma(N)=(\sigma(M),\sigma(F))$. In order to apply this construction to the NIEP, it is necessary
to determine when the matrix $N$ produced in this way is nonnegative. In
\cite{SmigocSubmatrixConstruction}, \v{S}migoc gives the following answer to this question:

For an $m\times p$ matrix $Y_1$, we define the sets:
\[
	\mathcal{L}(Y_1):=\{l\in\mathbb{R}^p:Y_1l\geq0\}
\]
and
\[
	\mathcal{K}(Y_1):=\{k\in\mathbb{R}^p:k^T=h^TY_1 \text{ for some nonnegative }
h\in\mathbb{R}^m \}.
\]
For a $p\times p$ matrix $C$ and an $m\times p$ matrix $Y_1$, we define $\mathcal{M}_n(Y_1,C)$ to be
the set of all $n\times n$ matrices
\[
	M=\left[
		\begin{array}{cc}
			A & K \\
			L & C
		\end{array}
	\right],
\]
such that $A$ is an $n_1\times n_1$ nonnegative matrix, $n=n_1+p$, every column of $L$ lies in
$\mathcal{L}(Y_1)$ and the transpose of every row of $K$ lies in $\mathcal{K}(Y_1)$.

\begin{thm}\textup{\bf\cite{SmigocSubmatrixConstruction}}\label{thm:NonnegativityConditions}
	Let the assumptions (i)--(iv) in Lemma \ref{lem:SubmatrixConstruction} hold. Assume also
that $B$ is nonnegative, that the Perron eigenvalue of $B$ lies in $\sigma(C)$ and that
$M\in\mathcal{M}_n(Y_1,C)$. Then the matrix $N$ of the lemma is nonnegative, i.e. the list
$(\sigma(M),\sigma(F))$ is realisable by a nonnegative matrix with principal submatrices $A$ and
$B$.
\end{thm}

Theorem \ref{thm:NonnegativityConditions} provides a method of producing new realisable lists from
old. With $p=1$, it allows us to replace the Perron eigenvalue of a known realisable list, for
example as in Theorem \ref{thm:SmigocDiagonalElement}. The $p=1$ case has been dealt with in detail in \cite{SmigocDiagonalElement}. With $p=2$, it allows us to replace the Perron eigenvalue and
another real eigenvalue, for example as in Theorem \ref{thm:SmigocSubmatrixConstruction}. The $p=2$ case is dealt with in \cite{SmigocSubmatrixConstruction} and we give further results in Section
\ref{sec:p=2}. With $p=3$, Theorem \ref{thm:NonnegativityConditions} allows us to replace the
Perron eigenvalue and a complex conjugate pair of eigenvalues (see Section \ref{sec:p=3}).

\section{A $p=2$ construction}\label{sec:p=2}

In this section, given a realisable list $(\rho,\lambda_2,\lambda_3,\ldots,\lambda_m)$,
where $\rho$ is the Perron eigenvalue and $\lambda_2$ is real, we present some lists
$(\mu_1,\mu_2,\ldots,\mu_n)$ such that $(\mu_1,\mu_2,\ldots,\mu_n,\lambda_3,\lambda_4,\ldots,\lambda_m)$ is realisable. This corresponds to letting $p=2$ in Lemma \ref{lem:SubmatrixConstruction}.

In \cite{SmigocSubmatrixConstruction}, \v{S}migoc characterises $\mathcal{L}(Y_1)$ and
$\mathcal{K}(Y_1)$ for the $p=2$ case. Using Lemma \ref{lem:e}, we may assume without loss of
generality that the eigenvector corresponding to $\rho$ is $e$. Let $z$ be a real
eigenvector corresponding to $\lambda_2$ and let $z_\mathrm{max}$ and $z_\mathrm{min}$ denote the maximal and minimal entries of $z$, respectively. In \cite{SmigocSubmatrixConstruction}, Section 4, \v{S}migoc shows that we may assume $z_\mathrm{max}>0$ and $z_\mathrm{min}\leq0$. She then gives the following characterisations of $\mathcal{L}(Y_1)$ and $\mathcal{K}(Y_1)$:

\begin{prop}\textup{\bf\cite{SmigocSubmatrixConstruction}}\label{prop:Cones1}
  If $z_\mathrm{max}>0$ and $z_\mathrm{min}<0$, then
\[ 
\mathcal{L}(Y_1)=\left\{\left[\begin{array}{c}l_1\\l_2\end{array}\right]:-\frac{l_1}{z_\mathrm{max}}
\leq
l_2\leq -\frac{l_1}{z_\mathrm{min}}\right\}.
\]
  If $z_\mathrm{max}>0$ and $z_\mathrm{min}=0$, then
\[ 
\mathcal{L}(Y_1)=\left\{\left[\begin{array}{c}l_1\\l_2\end{array}\right]:-\frac{l_1}{z_\mathrm{max}}
\leq
l_2 \text{ and } l_1\geq0\right\}.
\]
\end{prop}

\begin{prop}\textup{\bf\cite{SmigocSubmatrixConstruction}}\label{prop:Cones2}
  \[
  \mathcal{K}(Y_1)=\left\{\left[\begin{array}{c}k_1\\k_2\end{array}\right]:z_\mathrm{min}k_1\leq
k_2\leq z_\mathrm{max}k_1\right\}.
  \]
\end{prop}

We now give our $p=2$ construction.

\begin{lem}\label{lem:p=2}
Let the following assumptions hold:
\begin{enumerate}
\item[\textup{(i)}]
the list $\sigma_0:=(\rho,\lambda_2,\lambda_3,\ldots,\lambda_m)$ is realisable, where $\rho$ is
the Perron eigenvalue, $\lambda_2$ is real and $\rho\neq \lambda_2$;

\item[\textup{(ii)}]
$C'$ is a $2\times2$ matrix of the form
\[
  C':=\left[\begin{array}{cc}\gamma&1\\b_2&b_1+\gamma\end{array}\right],
\]
where $b_1$ is real, $\gamma=(\rho+\lambda_2-b_1)/2\geq0$ and
$b_2=\left((\rho-\lambda_2)^2-b_1^2\right)/4$;

\item[\textup{(iii)}]
$K':=[\begin{array}{cc}f&g\end{array}]$, where $f,g\in\mathbb{R}^{n-2}$, $g\geq0$ and
$f\geq(\gamma-\lambda_2)g$;

\item[\textup{(iv)}]
$L':=\left[\begin{array}{c}c^T\\d^T\end{array}\right]$, where $c,d\in\mathbb{R}^{n-2}$, $c\geq0$
and $d\geq(\rho-\gamma)c$;

\item[\textup{(v)}]
$A$ is an $(n-2)\times(n-2)$ nonnegative matrix;

\item[\textup{(vi)}]
$M'$ is the $n\times n$ matrix defined by
\[
	M':=\left[\begin{array}{cc}A&K'\\L'&C'\end{array}\right].
\]

\end{enumerate}
Then the list
$
	(\sigma(M'),\lambda_3,\lambda_4,\ldots,\lambda_m)
$
is realisable.
\end{lem}
\begin{proof}
  Let $B$ be a nonnegative matrix with spectrum $\sigma_0$. As in the construction of Lemma
\ref{lem:SubmatrixConstruction}, let $Y$ be an invertible matrix
such that
\[
  Y^{-1}BY=\left[\begin{array}{cc}C&*\\0&*\end{array}\right],
\]
where
\[
  C:=\left[\begin{array}{cc}\rho&0\\0&\lambda_2\end{array}\right].
\]
By Lemma \ref{lem:e}, we may assume without loss of generality that the Perron eigenvector of $B$
is $e$. Let $z$ be a real eigenvector of $B$ corresponding to $\lambda_2$, appropriately scaled so
that $z_\mathrm{max}=1$ and $z_\mathrm{min}\leq0$ (see the discussion preceding Proposition \ref{prop:Cones1}) and let us write $Y=[\begin{array}{cc}Y_1&Y_2\end{array}]$, where $Y_1=[\begin{array}{cc}e&z\end{array}]$.

Note that the definitions of $\gamma$ and $b_2$ assure
$
  \sigma(C')=(\rho,\lambda_2).
$
Therefore, since $\rho$ and $\lambda_2$ are distinct, we may diagonalise $C'$. Indeed, $X^{-1}C'X=C$, where
\[
  X:=\left[\begin{array}{cc}1&1\\\rho-\gamma&\lambda_2-\gamma\end{array}\right]
  \:\text{ and }\:
X^{-1}=\frac{1}{\rho-\lambda_2}\left[\begin{array}{cc}-\lambda_2+\gamma&1\\\rho-\gamma&-1\end{array}
\right].
\]
Now define
\[
  K:=K'X=\left[
    \begin{array}{cc}
      f+(\rho-\gamma)g & f+(\lambda_2-\gamma)g
    \end{array}
  \right],
\]
\[
  L:=X^{-1}L'=\frac{1}{\rho-\lambda_2}
  \left[
    \begin{array}{c}
      (-\lambda_2+\gamma)c^T+d^T \\
      (\rho-\gamma)c^T-d^T
    \end{array}
  \right]
\]
and
\[
  M:=\left[\begin{array}{cc}A&K\\L&C\end{array}\right].
\]
We will show that $M\in\mathcal{M}_n(Y_1,C)$ and that $M$ and $M'$ are similar (and hence cospectral). The result will then follow by Theorem \ref{thm:NonnegativityConditions}.

To see that $M\in\mathcal{M}_n(Y_1,C)$, we first note that since $g\geq0$ and
$f\geq(\gamma-\lambda_2)g\geq(\gamma-\rho)g$, we have
\[
  z_\mathrm{min}\left( f+(\rho-\gamma)g \right)\leq0\leq
f+(\lambda_2-\gamma)g\leq f+(\rho-\gamma)g
\]
 and hence, by Proposition \ref{prop:Cones2}, the transpose of every row of $K$ lies in
$\mathcal{K}(Y_1)$. Similarly, since $c\geq0$ and $d\geq(\rho-\gamma)c\geq(\lambda_2-\gamma)c$, we have that
\[
	-\left( (-\lambda_2+\gamma)c+d
\right)\leq(\rho-\gamma)c-d\leq0\leq-\frac{1}{z_\mathrm{min}}\left(
(-\lambda_2+\gamma)c+d \right),
\]
where the right-most inequality holds provided $z_\mathrm{min}\neq0$. Then, by Proposition
\ref{prop:Cones1}, every column of $L$ lies in $\mathcal{L}(Y_1)$.

Therefore, we have shown that $M\in\mathcal{M}_n(Y_1,C)$. Finally, it is easy to see that $M$ and $M'$ are similar:
\[
  M=\left[
    \begin{array}{cc}
      I & 0 \\
      0 & X
    \end{array}
  \right]^{-1}M'
  \left[
    \begin{array}{cc}
      I & 0 \\
      0 & X
    \end{array}
  \right].\qedhere
\]
\end{proof}

In the proof of Lemma \ref{lem:p=2}, we have shown that $M'$ is similar to a matrix in
$\mathcal{M}_n(Y_1,C)$. In the applications of this lemma, we will choose $A$, $K'$ and $L'$ in such
a way that $M'$ has a structure which makes its characteristic polynomial easy to compute. Several
such structured matrices---such as companion matrices, doubly companion matrices and block companion matrices---have been studied in the context of the NIEP, for example by Friedland, Laffey, \v{S}migoc and Cronin \cite{Friedland}, \cite{LSStructured}, \cite{Cronin} and indeed, the form of the matrix $C'$ in Lemma \ref{lem:p=2} has been chosen with such matrices in mind.

For example, letting
\begin{equation}\label{eq:AForCompanion}
  A=\left[\begin{array}{cccc}
    \gamma & 1 & &\\
    & \gamma & \ddots &\\
    & & \ddots & 1\\
    & & & \gamma
  \end{array}\right],
\end{equation}
$d\geq0$, $f=[\begin{array}{ccccc}0&0&\cdots&0&1\end{array}]^T$ and $c=g=0$, the matrix $M'$ becomes a companion matrix plus a scalar and as such, the characteristic polynomial of $M'$ is easy to write down. The case where $M'$ is a companion matrix plus a scalar is developed formally in Theorem \ref{thm:Basicp=2}.

Alternatively, keeping $c$, $d$, $f$ and $g$ as above, but setting
\[
	A=\left[
		\begin{array}{ccccccc}
			\gamma & 1 &  &  &  &  &  \\
			& \ddots & \ddots &  &  &  &  \\
			&  & \gamma & 1 &  &  &  \\
			* & * & \cdots & * & 1 &  &  \\
			&  &  &  & \gamma & \ddots &  \\
			&  &  &  &  & \ddots & 1 \\
			&  &  &  &  &  & \gamma \\
		\end{array}
	\right],
\]
the matrix $M'$ becomes a 2-block companion matrix plus a scalar.

Taking $f$, $g$ and $d$ as above, $c=[\begin{array}{ccccc}*&0&0&\cdots&0\end{array}]^T$ and
\[
	A=\left[
		\begin{array}{cccc}
			* & 1 & & \\
			* & \gamma & \ddots & \\
			\vdots & & \ddots & 1 \\
			* & & & \gamma
		\end{array}
	\right],
\]
then $M'$ becomes a doubly companion matrix plus a scalar.

\begin{ex}\label{ex:DoublyCompanion}
	Let $\sigma$ be any list such that $(8,2,\sigma)$ is realisable. In Lemma \ref{lem:p=2}, let
us take $\rho=8$, $\lambda_2=2$, $b_1=10$ and $n=4$. It is easily verified that the matrices
	\[
		K':=\left[ \begin{array}{cc} 0 & 0\\ 1 & 0 \end{array} \right],
		\:\:
		L':=\left[ \begin{array}{cc} 42 & 0\\ 336 & 28 \end{array} \right]
		\:\text{ and }\:
		A:=\left[ \begin{array}{cc} 0 & 1\\ 3 & 0 \end{array} \right]
	\]
	satisfy the hypotheses of the lemma and the matrix $M'$ of the lemma then becomes
	\[
		M'=\left[
\begin{array}{cccc}
 0 & 1 & 0 & 0 \\
 3 & 0 & 1 & 0 \\
 42 & 0 & 0 & 1 \\
 336 & 28 & -16 & 10
\end{array}
\right].
	\]
	$M'$ is a doubly companion matrix with characteristic polynomial
	\begin{eqnarray*}
		w(x)&=&x^4-10 x^3+13 x^2-40 x+36\\
		&=& (x-9)(x-1)(x^2+4)
	\end{eqnarray*}
	and hence the list $(9,1,2i,-2i,\sigma)$ is realisable.
\end{ex}

\begin{ex}
	Let $\sigma$ be any list such that $(8,-2,\sigma)$ is realisable. In Lemma \ref{lem:p=2},
take $\rho=8$, $\lambda_2=-2$, $b_1=6$ and $n=7$. Then the matrix
\[
	M'=\left[\begin{array}{cc}A&K'\\L'&C'\end{array}\right]=
\begin{pmat}[{....|.}]
 0 & 1 & 0 & 0 & 0 & 0 & 0 \cr
 0 & 0 & 1 & 0 & 0 & 0 & 0 \cr
 \frac{296}{29} & \frac{5}{29} & 0 & 1 & 0 & 0 & 0 \cr
 0 & 0 & 0 & 0 & 1 & 0 & 0 \cr
 0 & 0 & 0 & 0 & 0 & 1 & 0 \cr\-
 0 & 0 & 0 & 0 & 0 & 0 & 1 \cr
 \frac{17024}{29} & \frac{30016}{29} & \frac{15872}{29} & 0 & 0 & 16 & 6 \cr
\end{pmat}
\]
satisfies the hypotheses of the lemma. $M'$ is an example of a 2-block companion matrix. Its
characteristic polynomial is
\[
	w(x)=\frac{1}{29}\left(29 x^2-203 x-266\right)\left(x^4+64\right)(x+1)
\]
and hence the list
\[
	(8.128\ldots,-1.128\ldots,2+2i,2-2i,-2+2i,-2-2i,-1,\sigma)
\]
is realisable.
\end{ex}

\begin{thm}\label{thm:Basicp=2}
Let the list $\sigma_0:=(\rho,\lambda_2,\lambda_3,\ldots,\lambda_m)$ be realisable, where $\rho$ is
the Perron eigenvalue, $\lambda_2$ is real and $\rho\neq \lambda_2$. Let $b_1$ be any real number such that
\begin{equation}\label{eq:GammaCondition}
	\gamma:=\frac{\rho+\lambda_2-b_1}{2}\geq0,
\end{equation}
let
\begin{equation}\label{eq:b2Condition}
	b_2:=\frac{(\rho-\lambda_2)^2-b_1^2}{4}
\end{equation}
and let $b_3,b_4,\ldots,b_n$ be any nonnegative numbers. Then the list
\[
	(\mu_1,\mu_2,\ldots,\mu_n,\lambda_3,\lambda_4,\ldots,\lambda_m)
\]
is realisable, where $\mu_1,\ldots,\mu_n$ are the roots of the polynomial
\[
  w(x):=(x-\gamma)^n-b_1(x-\gamma)^{n-1}-b_2(x-\gamma)^{n-2}-\cdots-b_{n-1}(x-\gamma)-b_n.
\]
\end{thm}
\begin{proof}
	In Lemma \ref{lem:p=2}, let $A$ be as in (\ref{eq:AForCompanion}) and let
$d=[\begin{array}{cccc}b_n&b_{n-1}&\cdots&b_3\end{array}]^T$, $f=[\begin{array}{ccccc}0&0&\cdots&0&1\end{array}]^T$ and $c=g=0$. Then, note that $M'-\gamma I_n$ becomes a companion matrix (where $M'$ is defined in the statement of the lemma) and as such it has characteristic polynomial $w(x+\gamma)$. Hence $M'$ has characteristic polynomial $w(x)$.
\end{proof}

\begin{ex}\label{ex:p=2a}
	Let $\sigma$ be any list such that $(4,2,\sigma)$ is realisable. Taking $\rho=4$ and
$\lambda_2=2$ in Theorem \ref{thm:Basicp=2}, let us choose $n=4$, $b_1=6$, $b_3=10$ and $b_4=25$. Then, the polynomial $w(x)$ of the theorem becomes
	\begin{eqnarray*}
		w(x)&=&x^4-6x^3+8x^2-10x-25\\
		&=&(x-5)(x^2-2x+5)(x+1)
	\end{eqnarray*}
	and so the list $(5,1+2i,1-2i,-1,\sigma)$ is realisable.
\end{ex}

At this point, we wish to use Theorem \ref{thm:LS} in conjunction with Theorem \ref{thm:Basicp=2} to
produce a class of spectra which may replace the eigenvalues $\rho$ and $\lambda_2$; however,
Theorem \ref{thm:LS} deals with realisation by matrices of the form $G+\gamma I_n$, where $G$ has trace zero and so applying this directly would correspond to taking $b_1=0$ in Theorem \ref{thm:Basicp=2}. With this in mind, we will present a slight modification of Theorem \ref{thm:LS}, in which we examine realisation by a matrix of the form $G+\gamma I_n$, where $G$ may have nonzero trace. First, we will require a lemma from \cite{LaffeySmigoc}:

\begin{lem}\textup{\bf\cite{LaffeySmigoc}}\label{lem:LS}
	Let $b_1\geq0$ and let $(\lambda_2,\lambda_3,\ldots,\lambda_n)$ be a list of complex numbers, closed under complex conjugation and with nonpositive real parts. Set $\rho:=b_1-\lambda_2-\lambda_3-\cdots-\lambda_n$ and
	\[
		f(x):=(x-\rho)\prod_{i=2}^n(x-\lambda_i)=x^n-b_1x^{n-1}-b_2x^{n-2}-\cdots-b_n.
	\]
	Then $b_2\geq0$ implies $b_i\geq0$ for all $i=3,4,\ldots,n$.
\end{lem}

\begin{thm}\label{thm:NewLS}
  Let $\sigma:=(\rho,\lambda_2,\lambda_3,\ldots,\lambda_n)$ be realisable, where
$\rho$ is the Perron eigenvalue and $\mathrm{Re}\,\lambda_i\leq0$ for all $i=2,3,\ldots,n$. Then for any nonnegative number $b_1$ with $b_1\leq s_1(\sigma)$ and $(n-1)b_1^2\leq ns_2(\sigma)-s_1(\sigma)^2$, $\sigma$ may be realised by a matrix of the form $G+\gamma I_n$, where $G$ is a nonnegative companion matrix with trace $b_1$ and $\gamma$ is a nonnegative scalar.
\end{thm}
\begin{proof}
  Since $\sigma$ is realisable, note that $s_1(\sigma)\geq0$ and the JLL condition $s_1(\sigma)^2\leq ns_2(\sigma)$ holds. Choose any nonnegative $b_1$ such that $b_1\leq s_1(\sigma)$ and $(n-1)b_1^2\leq
ns_2(\sigma)-s_1(\sigma)^2$. Let $\gamma:=(s_1(\sigma)-b_1)/n$,
\[\sigma':=(\rho-\gamma,\lambda_2-\gamma,\lambda_3-\gamma,\ldots,\lambda_n-\gamma)\]
  and
  \[
    g(x):=(x-\rho+\gamma)\prod_{i=2}^n(x-\lambda_i+\gamma).
  \]
  It is clear from the definition of $\gamma$ that $s_1(\sigma')=b_1$. Therefore, we may write
$g(x)$
as
  \[
    g(x)=x^n-b_1x^{n-1}-b_2x^{n-2}-\cdots-b_n.
  \]
  Now, the elements of $\sigma'$ are the roots of $g$ and hence, using Newton's Identities for the roots of a polynomial, we have that
  \begin{eqnarray*}
    b_2 &=& \frac{1}{2}\left( s_2(\sigma')-b_1^2 \right) \\
    &=& \frac{1}{2}\left( s_2(\sigma)-2\gamma s_1(\sigma)+n\gamma^2-b_1^2 \right) \\
    &=& \frac{1}{2n}\left( n s_2(\sigma)-s_1(\sigma)^2-(n-1) b_1^2 \right) \\
    &\geq& 0.
  \end{eqnarray*}
  The complex numbers $\lambda_2-\gamma,\lambda_3-\gamma,\ldots,\lambda_n-\gamma$ have nonpositive real parts and hence by Lemma \ref{lem:LS}, $b_i\geq0$ for all $i=3,4,\ldots,n$.
Therefore, the companion matrix of $g$, $G$ say, is nonnegative, has trace $b_1$ and has spectrum
$\sigma'$. It follows that $G+\gamma I_n$ has spectrum $\sigma$.
\end{proof}

\begin{rem}
	Similarly to the remark following Theorem \ref{thm:LS}, we note that, in the proof of Theorem \ref{thm:NewLS}, it was only required that
$\lambda_2-\gamma,\lambda_3-\gamma,\ldots,\lambda_n-\gamma$ have nonpositive real parts. Therefore, the condition that $\mathrm{Re}\,\lambda_i\leq0$ for all
$i=2,3,\ldots,n$ in the statement of the theorem can be relaxed to
$\mathrm{Re}\,\lambda_i\leq(s_1(\sigma)-b_1)/n$.
\end{rem}

\begin{thm}\label{thm:p=2}
	Let $\sigma_0:=(\rho,\lambda_2,\lambda_3,\ldots,\lambda_m)$ be realisable, where
$\rho$ is the Perron eigenvalue, $\lambda_2$ is real and $\rho\neq \lambda_2$. Let
	\begin{equation}\label{eq:deltaRange}
		(n-2)\max\{0,\lambda_2\}\leq\delta\leq \frac{1}{2}(n-2)(\rho+\lambda_2)
	\end{equation}
	and let $\mu:=(\mu_1,\mu_2,\ldots,\mu_n)$ be a list of complex numbers, closed under complex conjugation, with $\mu_1\geq0$ and
$\text{Re}\,\mu_i\leq \delta/(n-2)$ for all $i=2,3,\ldots,n$. Assume also that
	\begin{equation}\label{eq:s1def}
		s_1(\mu)=\rho+\lambda_2+\delta
	\end{equation}
	and
	\begin{equation}\label{eq:s2def}
		s_2(\mu)=\rho^2+\lambda_2^2+\frac{\delta ^2}{n-2}.
	\end{equation}
	Then the list
	$
		(\mu_1,\mu_2,\ldots,\mu_n,\lambda_3,\lambda_4,\ldots,\lambda_m)
	$
	is realisable.
\end{thm}
\begin{proof}
	We will show that $\mu$ is the spectrum of a nonnegative matrix of the form
	\begin{equation}\label{eq:RequiredForm}
      \begin{pmat}[{...|.}]
      \gamma & 1 & & & 0 & 0 \cr
      & \gamma & \ddots & & \vdots & \vdots \cr
      & & \ddots & 1 & 0 & 0 \cr
      & & & \gamma & 1 & 0 \cr\-
      0 & 0 & \cdots & 0 & \gamma & 1 \cr
      b_n & b_{n-1} & \cdots & b_3 & b_2 & b_1+\gamma \cr
   \end{pmat},
	\end{equation}
	where $\gamma$ and $b_2$ satisfy (\ref{eq:GammaCondition}) and (\ref{eq:b2Condition}), respectively. The result will then follow by Theorem \ref{thm:Basicp=2}.

	To see that $\mu$ is realisable, from Theorem \ref{thm:LS} and the remark that follows it, it suffices to check that $s_1(\mu)^2\leq ns_2(\mu)$ and that $\text{Re}\,\mu_i\leq s_1(\mu)/n$ for all $i=2,3,\ldots,n$. For the first of these two conditions, consider $ns_2(\mu)-s_1(\mu)^2$ as a quadratic in $\delta$:
	\[
ns_2(\mu)-s_1(\mu)^2=\frac{2}{n-2}\delta^2-2(\rho+\lambda_2)\delta+(n-1)(\rho^2+\lambda_2^2)-2\rho
\lambda_2.
	\]
The coefficient of $\delta^2$ in this quadratic is positive and its discriminant is
	\[
		-\frac{4 n \left(\rho-\lambda_2\right)^2}{n-2}<0.
	\]
	Therefore, as required, $ns_2(\mu)-s_1(\mu)^2>0$ for all real $\delta$. For the second condition, let
	\begin{equation}\label{eq:b1def}
		b_1:=\rho+\lambda_2-\frac{2 \delta }{n-2}.
	\end{equation}
	For all $\delta$ satisfying (\ref{eq:deltaRange}), we have $0\leq b_1\leq s_1(\mu)$ and equations (\ref{eq:s1def}) and (\ref{eq:b1def}) then give
	\[
		\text{Re}\,\mu_i\leq\frac{\delta}{n-2}=\frac{s_1(\mu)-b_1}{n}\leq\frac{s_1(\mu)}{n},
	\]
	as required and so $\mu$ is realisable.
	
Furthermore, since
\[
	(n-2)\lambda_2\leq\delta\leq\frac{1}{2}(n-2)(\rho+\lambda_2)\leq(n-2)\rho,
\]
we have that
	\[
		ns_2(\mu)-s_1(\mu)^2-(n-1)b_1^2=\frac{2 n\left(\delta -(n-2)
\lambda_2\right)\left((n-2)
\rho-\delta \right) }{(n-2)^2}\geq0,
	\]
	so $b_1$ satisfies the conditions imposed on it by Theorem \ref{thm:NewLS}. Hence, by Theorem \ref{thm:NewLS} and the remark that follows it, $\mu$ may be realised by
a nonnegative matrix of the form (\ref{eq:RequiredForm}) and so $\mu_1,\mu_2,\ldots,\mu_n$ are the
roots of a polynomial of the form
	\[
		w(x):=(x-\gamma)^n-b_1(x-\gamma)^{n-1}-b_2(x-\gamma)^{n-2}-\cdots-b_{n-1}(x-\gamma)-b_n,
	\]
	where
	\begin{equation}\label{eq:gamma=}
		\gamma=\frac{s_1(\mu)-b_1}{n}.
	\end{equation}
	
	So it remains to show that $\gamma$ and $b_2$ satisfy (\ref{eq:GammaCondition}) and (\ref{eq:b2Condition}). To see this, consider the list
	\[
		\mu':=(\mu_1-\gamma,\mu_2-\gamma,\ldots,\mu_n-\gamma)
	\]
	and the polynomial
	\[
		w'(x):=x^n-b_1x^{n-1}-b_2x^{n-2}-\cdots-b_{n-1}x-b_n.
	\]
	The elements of $\mu'$ are the roots of $w'$ and so, using Newton's Identities for the roots of a polynomial, we have that
  \begin{align}
    b_2 &\:=\: \frac{1}{2}\left( s_2(\mu')-b_1^2 \right) \notag \\
    &\:=\: \frac{1}{2}\left( s_2(\mu)-2\gamma s_1(\mu)+n\gamma^2-b_1^2 \right).\label{eq:b2=}
  \end{align}
  Now, by eliminating $\delta$ from (\ref{eq:s1def}) and (\ref{eq:b1def}), we see that
  \begin{equation}\label{eq:s1=}
  	s_1(\mu)=\frac{n(\rho+\lambda_2)-(n-2) b_1}{2}
  \end{equation}
  and by eliminating $\delta$ from (\ref{eq:s2def}) and (\ref{eq:b1def}), we have
  \begin{equation}\label{eq:s2=}
  	s_2(\mu)=\rho^2+\lambda_2^2+\frac{1}{4} (n-2) \left(\rho+\lambda_2-b_1\right)^2.
  \end{equation}
  Substituting (\ref{eq:s1=}) in (\ref{eq:gamma=}), we obtain (\ref{eq:GammaCondition}) (the fact that $\gamma$ is nonnegative is easily seen from (\ref{eq:b1def})) and then,
substituting (\ref{eq:s1=}), (\ref{eq:s2=}) and (\ref{eq:GammaCondition}) into (\ref{eq:b2=}) gives (\ref{eq:b2Condition}).
  
  Finally, from Theorem \ref{thm:Basicp=2}, we conclude that
	\[
		(\mu_1,\mu_2,\ldots,\mu_n,\lambda_3,\lambda_4,\ldots,\lambda_m)
	\]
	is realisable.
\end{proof}

\begin{ex}\label{ex:p=2b}
  Let $\sigma$ be any list such that $(1,0,\sigma)$ is realisable. Letting $\rho=1$, $\lambda_2=0$,
$n=4$ and $\delta=0$ in Theorem \ref{thm:p=2}, we see that the list
$(\mu_1,\mu_2,\mu_3,\mu_4,\sigma)$ is also realisable, provided $\mu_1\geq0$,
$(\mu_2,\mu_3,\mu_4)$ is closed under complex conjugation,
$\text{Re}\,\mu_2,\text{Re}\,\mu_3,\text{Re}\,\mu_4\leq0$ and
$\sum_{i=1}^4\mu_i=\sum_{i=1}^4\mu_i^2=1$. For example,
\[
	\left(\frac{1+\sqrt{5}}{2},\frac{1-\sqrt{5}}{2},i,-i,\sigma\right)
\]
is realisable.
\end{ex}

\begin{ex}\label{ex:p=2c}
	Let $\sigma$ be any list such that $(1,-1,\sigma)$ is realisable. Letting $\rho=1$,
$\lambda_2=-1$ and $\delta=0$ in Theorem \ref{thm:p=2}, we have that for any $n\geq3$, the list
\[
	(\rho,\underbrace{-\lambda,-\lambda,\ldots,-\lambda}_{n-1 \text{ eigenvalues}},\sigma)
\]
is realisable, where
\[
	\rho:=\sqrt{\frac{2(n-1)}{n}} \:\text{ and }\: \lambda:=\sqrt{\frac{2}{n(n-1)}}.
\]

Alternatively (again taking $\delta=0$), for any $m\in\mathbb{N}$, Theorem \ref{thm:p=2} also gives
that the list
\[
	\left(
\sqrt{2},\underbrace{-\frac{1}{\sqrt{2}m}\pm\frac{1}{\sqrt{2}m}i,\ldots,-\frac{1}{\sqrt{2}m}\pm\frac
{1}{\sqrt{2}m}i}_{m \text{ pairs}},\sigma \right)
\]
is realisable.
\end{ex}

\begin{rem}
	In Examples \ref{ex:p=2b} and \ref{ex:p=2c}, it was possible to construct a new realisable list with the same trace as the original list. This was made possible by the fact that $\lambda_2\leq0$ in both cases and thus we could choose $\delta=0$ in Theorem \ref{thm:p=2}; however, even when $\lambda_2>0$, it may be possible to preserve the trace of the original spectrum using Theorem \ref{thm:Basicp=2} (see Example \ref{ex:p=2a}).
\end{rem}

\section{A $p=3$ construction}\label{sec:p=3}

In this section, we let $p=3$ in Lemma \ref{lem:SubmatrixConstruction}. For ease of calculation of the characteristic polynomial of $M$, we will confine our attention to the case where $n_1=1$ and so $M$ is a $4\times 4$ matrix. In this case, we seek to replace the eigenvalues $\rho,\alpha+i\beta,\alpha-i\beta$ of a realisable list with eigenvalues $\mu_1,\mu_2,\mu_3,\mu_4$, where $\sigma(M)=(\mu_1,\mu_2,\mu_3,\mu_4)$.

\begin{thm}\label{thm:p=3}
	Let the list $\sigma_0:=(\rho, \alpha+i\beta,\alpha-i\beta,\lambda_4,\lambda_5,\ldots,\lambda_m)$ be realisable, where $\rho$ is the Perron eigenvalue, $\alpha$ is real and $\beta>0$. Let $a$, $t$, and
$\eta$ be any real numbers satisfying $a,t\geq0$ and $0<\eta\leq1$. Then the list
	\[
		\sigma_1:=(\mu_1,\mu_2,\mu_3,\mu_4,\lambda_4,\lambda_5,\ldots,\lambda_m)
	\]
	is realisable, where $\mu_1,\mu_2,\mu_3,\mu_4$ are the roots of the polynomial
	\begin{equation}\label{eq:q(x)WithEta}
		q(x):=(x-\rho )\left((x-\alpha )^2+\beta ^2\right)(x-a) -t \left((x-\alpha )
((1+\eta ) x-\alpha -\eta  \rho )+\beta ^2\right).
	\end{equation}
\end{thm}
\begin{proof}
	Let the assumptions (i) and (ii) in Lemma \ref{lem:SubmatrixConstruction} hold, where $B$ is
a nonnegative matrix with spectrum $\sigma_0$ and
	\[
		C=\left[
			\begin{array}{ccc}
				\rho & 0 & 0 \\
				0 & \alpha & \beta \\
				0 & -\beta & \alpha
			\end{array}
		\right].
	\]
	By Lemma \ref{lem:e}, we may assume without loss of generality that the eigenvector
corresponding to $\rho$ is $e$ and so we may write
	\[
		Y_1=\left[
			\begin{array}{ccc}
				e & u & v \\
			\end{array}
		\right],
	\]
	where $u$ and $v$ are real vectors and $u\pm iv$ are eigenvectors corresponding to the eigenvalues $\alpha\pm i\beta$, respectively. We may also assume that
	\[
		\eta=u_1^2+v_1^2=\max_i(u_i^2+v_i^2).
	\]
	To see this, suppose instead that $\tau=u_k^2+v_k^2=\max_i(u_i^2+v_i^2)$. Then we may replace $B$ by $PBP^T$ and $Y$ by $PYD$, where $P$ is the permutation matrix obtained by swapping rows 1 and $k$ of $I_m$ and $D$ is the diagonal matrix
	\[
		D:=\left[
			\begin{array}{ccccccc}
				1 &  &  &  &  &  &  \\
				& \sqrt{\eta/\tau}  &  &  &  &  &  \\
				&  & \sqrt{\eta/\tau} &  &  &  &  \\
				&  &  & 1 &  &  &  \\
				&  &  &  & 1 &  &  \\
				&  &  &  &  & \ddots &  \\
				&  &  &  &  &  & 1
			\end{array}
		\right].
	\]
	
	Now consider the matrix
	\[
		M:=\begin{pmat}[{|..}]
			a & t & tu_1 & tv_1 \cr\-
			1 & & & \cr
			u_1 & & C & \cr
			v_1 & & & \cr
		\end{pmat}.
	\]
	For all $i=1,2,\ldots,m$, the Cauchy-Schwarz inequality gives
	\[
		|u_iu_1+v_iv_1|
		\leq \sqrt{(u_i^2+v_i^2)(u_1^2+v_1^2)}
		\leq \eta
		\leq 1
	\]
	and therefore $-(u_iu_1+v_iv_1)\leq1$. Now, since $1+u_iu_1+v_iv_1$ is precisely the
$i^\mathrm{th}$ component of the vector
	\[
		Y_1\left[
			\begin{array}{c}
				1\\
				u_1\\
				v_1
			\end{array}
		\right],
	\]
	we see that
	\[
		\left[
			\begin{array}{c}
				1\\
				u_1\\
				v_1
			\end{array}
		\right]\in\mathcal{L}(Y_1).
	\]
	Furthermore, since
	\[
		\left[\begin{array}{ccc}t & tu_1 & tv_1\end{array}\right]
		=
		\left[\begin{array}{ccccc}t & 0 & 0 & \cdots & 0\end{array}\right]
		Y_1,
	\]
	we have that
	\[
		\left[\begin{array}{ccc}t & tu_1 & tv_1\end{array}\right]\in\mathcal{K}(Y_1).
	\]
	Therefore $M\in \mathcal{M}_n(Y_1,C)$ and so by Theorem \ref{thm:NonnegativityConditions},
the list
	\[
		(\sigma(M),\lambda_4,\ldots,\lambda_m)
	\]
	is realisable.
	
	Finally, the characteristic polynomial of $M$ is
	\begin{multline*}
		q(x)=(x-\rho ) \left((x-\alpha )^2+\beta ^2\right)(x-a) \\  -t \left((x-\alpha )
\left(\left(1+u_1^2+v_1^2\right) x-\alpha -\left(u_1^2+v_1^2\right) \rho \right)+\beta ^2\right),
	\end{multline*}
	which, after the substitution $u_1^2+v_1^2=\eta$, becomes the polynomial mentioned in the
statement of the theorem.
\end{proof}

\begin{ex}
	Consider the list
	\[
		\sigma_0:=(26,-12+2i,-12-2i,-1+14i,-1-14i).
	\]
	We have $s_1=0$ and $s_2=566$, so $\sigma_0$ is realisable by Theorem \ref{thm:LS}. Applying
Theorem \ref{thm:p=3} with $\rho=26$, $\alpha=-12$, $\beta=2$, $a=0$, $\eta=1$ and $t=550$, we
obtain the new realisable list
	\[
		(42.7876\ldots,5.17729\ldots,-11.9818\ldots,-33.9831\ldots,-1+14i,-1-14i).
	\]
	If desired, we may use three applications of Theorem \ref{thm:Guo} to round off these
numbers and produce
	\[
		(43,5,-12,-34,-1+14i,-1-14i).
	\]
	Like $\sigma_0$, this list is extreme in the sense that it is not realisable for any
smaller Perron eigenvalue (it has trace 0).
\end{ex}

In order to see what type of spectra may be obtained from Theorem \ref{thm:p=3}, we need to analyse the polynomial $q(x)$ in (\ref{eq:q(x)WithEta}). First, we note that for $t=0$, $\sigma_1$ differs from $\sigma_0$ only by the addition of the nonnegative eigenvalue $a$. Therefore, in what follows, we will always assume that $a\leq\rho$ and hence $\rho$ will remain the Perron eigenvalue of $\sigma_1$ after the addition of $a$ to the list. We will now examine how $\sigma_1$ varies as we increase $t$.

To investigate the roots of $q(x)$, it is convenient to label
\begin{eqnarray*}
	f(x) &=& (x-\rho ) \left((x-\alpha )^2+\beta ^2\right) (x-a),\\
	g(x) &=& (x-\alpha ) ((1+\eta ) x-\alpha -\eta  \rho )+\beta ^2,
\end{eqnarray*}
so that $q(x)=f(x)-tg(x)$. As $t$ approaches infinity, the quadratic, linear and constant terms of
$q(x)$ become increasingly dominated by those of $-tg(x)$ and therefore two of the roots of $q(x)$, say $\mu_+$ and $\mu_-$, will approach those of $g(x)$; however, as $\eta$ tends to zero,
$g(x)\rightarrow(x-\alpha)^2+\beta^2$ and so for small $\eta$, the eigenvalues $\alpha\pm i\beta$
of $\sigma_0$ will exhibit little variation as $t$ increases. Therefore, from now on, we will always
assume that $\eta=1$. Under this assumption, we rewrite:
\begin{align}
  g(x)&= \beta ^2+(x-\alpha ) (2x-\alpha -\rho ),\notag\\
  q(x)&=(x-\rho ) \left((x-\alpha )^2+\beta ^2\right) (x-a)-t \left(\beta ^2+(x-\alpha ) (2
x-\alpha -\rho )\right)\label{eq:q(x)}
\end{align}
and the roots of $g(x)$ become
\begin{align}
	\lambda_+ &\::=\:\frac{1}{4} \left(\rho +3 \alpha+\sqrt{(\rho-\alpha)^2-8 \beta
^2}\right),\notag\\
	\lambda_- &\::=\:\frac{1}{4} \left(\rho +3 \alpha-\sqrt{(\rho-\alpha)^2-8 \beta
^2}\right).\label{eq:lambdas}\\\notag
\end{align}

We now examine how the Perron eigenvalue of $\sigma_1$ depends on $t$. Let $s\geq0$. Substituting $\rho+s$ for $x$ in (\ref{eq:q(x)}) and solving for $t$ yields
\begin{equation}\label{eq:t(s)}
	t=\frac{s (\rho +s-a) \left(\beta ^2+(\rho +s-\alpha )^2\right)}{\beta ^2+(\rho +s-\alpha )
(\rho +2 s-\alpha )},
\end{equation}
so we see that for large $s$, $s\sim\sqrt{t}$.

To sum up, let us denote the roots of $q(x)$ by $\rho+s,\mu_+,\mu_-,\psi$, where $\rho+s$ is the
Perron eigenvalue of $\sigma_1$ and $\psi$ is the remaining real root. We have observed that
$s\rightarrow\infty$ and $|\mu_\pm-\lambda_\pm|\rightarrow0$ as $t\rightarrow\infty$. Finally, we
note that the matrix $M$ in the proof of Theorem \ref{thm:p=3} has trace $\rho+2\alpha+a$ (i.e.
$\mathrm{trace}(\sigma_1)=\mathrm{trace}(\sigma_0)+a$) and in particular, this trace is independent
of $t$. Thus, we must have that $\psi\rightarrow-\infty$ as $t\rightarrow\infty$ and
$\psi\sim-\sqrt{t}$ for large $t$.

Since two of the eigenvalues of the spectrum $\sigma_1$ converge to $\lambda_\pm$ as $t$ increases, it is useful to examine how $\lambda_\pm$ depend on the initial eigenvalues $\rho$ and $\alpha\pm i\beta$. Consider the following conditions:
\begin{eqnarray}
  &&\rho\geq\alpha+2\sqrt{2}\beta\label{eq:ineq1};\\
  &&\alpha<0 \text{ and } \rho\geq-(\alpha^2+\beta^2)/\alpha\label{eq:ineq2};\\
  &&\rho\geq-3\alpha.\label{eq:ineq3}
\end{eqnarray}
From the formulae for $\lambda_\pm$ (\ref{eq:lambdas}), we see that $\lambda_+$ and $\lambda_-$ are real when (\ref{eq:ineq1}) holds and complex otherwise. Assuming $\lambda_+$ and $\lambda_-$ are real, they have different sign ($\lambda_-\leq0\leq\lambda_+$) when (\ref{eq:ineq2}) holds and the same sign otherwise. Assuming $\lambda_+$ and $\lambda_-$ are real and have equal sign, $\lambda_+,\lambda_-\geq0$ when (\ref{eq:ineq3}) holds and $\lambda_+,\lambda_-\leq0$ otherwise. Figure \ref{fig:LambdaDiagram} illustrates these various possibilities.

\begin{figure}[h]
	\centering
	\includegraphics[scale=0.98]{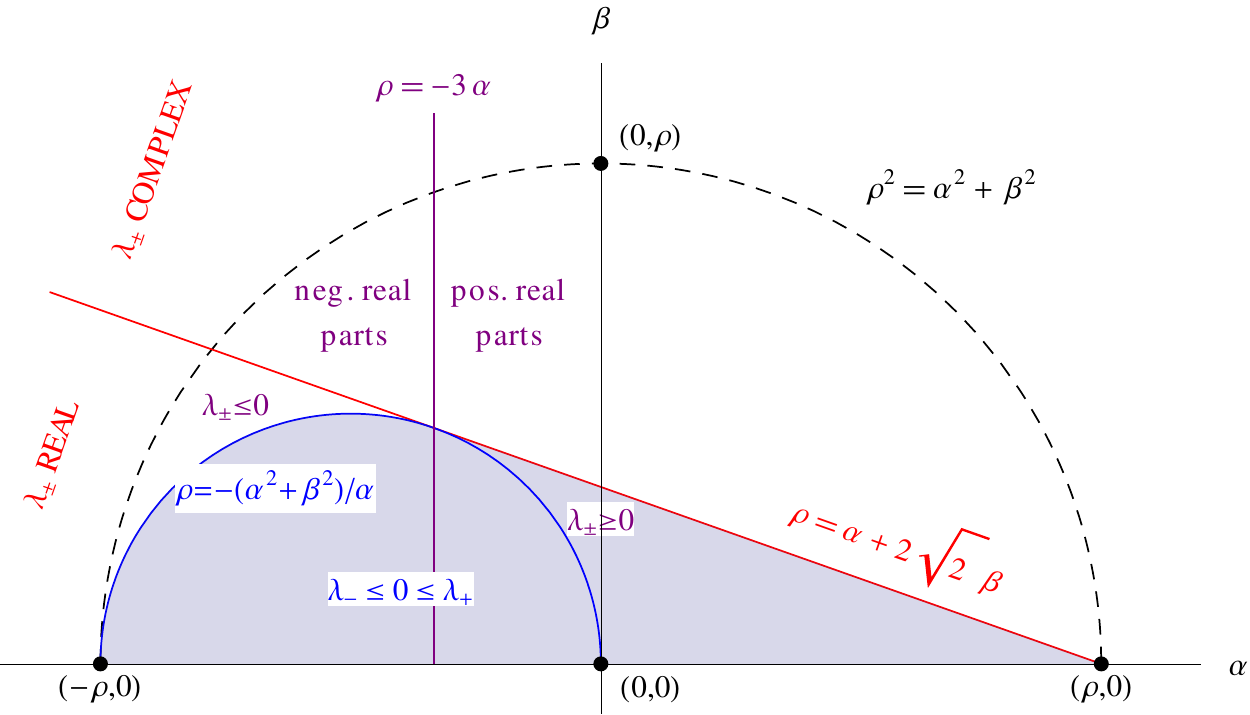}
	\caption{Dependence of $\lambda_+$ and $\lambda_-$ on $\rho$, $\alpha$ and $\beta$}
	\label{fig:LambdaDiagram}
\end{figure}

In general, the roots of $q(x)$ are complicated functions of $\rho,\alpha,\beta,a$ and $t$,
but there is a situation where these formulae may be simplified. Let us consider the case where
(\ref{eq:ineq1}) holds and either (\ref{eq:ineq2}) or (\ref{eq:ineq3}) holds. This case corresponds to the shaded region of Figure \ref{fig:LambdaDiagram}. Under these assumptions, $\lambda_+\geq0$ and this allows us to set $a=\lambda_+$. Hence $x-\lambda_+$ becomes a factor of $q(x)$. Similarly to the substitution made in (\ref{eq:t(s)}), we may then specify a value of $t$ which forces the remaining cubic polynomial to have the root $\rho+s$ and we may then factor out $x-\rho-s$. Finally, the remaining quadratic may be solved, giving the following result:

\begin{prop}\label{prop:RootsExpression}
	Let the list $(\rho, \alpha+i\beta,\alpha-i\beta,\lambda_4,\lambda_5,\ldots,\lambda_m)$ be realisable,
where $\rho$ is the Perron eigenvalue, $\alpha$ is real and $\beta>0$. Assume also that either
(\ref{eq:ineq2}) holds or both (\ref{eq:ineq1}) and (\ref{eq:ineq3}) hold. Then for all $s\geq0$,
the list
	\[
		(\rho+s,\mu_+,\mu_-,\lambda_+,\lambda_4,\lambda_5,\ldots,\lambda_m)
	\]
	is realisable, where
	\begin{multline*}
		\mu_\pm=\\ \alpha -\frac{s}{2}\pm\frac{1}{2}\sqrt{\frac{ s (s-2 \alpha
)^2+\left(s^2-4 s \alpha -4 \beta ^2\right) \rho +\left(4 \beta ^2+s (3 s-4 \alpha +4 \rho )\right)
\lambda _-}{\left(s+\rho -\lambda _-\right)}}
	\end{multline*}
	and $\lambda_+$ and $\lambda_-$ are defined in (\ref{eq:lambdas}).
\end{prop}
\begin{proof}
  From the preceding discussion, it suffices to show that (\ref{eq:ineq2}) implies
(\ref{eq:ineq1}). Indeed
\[
\frac{\alpha^2+\beta^2}{-\alpha}-(\alpha+2\sqrt{2}\beta)=\frac{(\sqrt{2}\alpha+\beta)^2}{-\alpha}
\geq0.\qedhere
\]
\end{proof}

\begin{ex}
  Let $\sigma$ be any list such that
\[
  \sigma_0:=(6,-2+2\sqrt{2}i,-2-2\sqrt{2}i,\sigma)
\]
is realisable. Substituting $\rho=6$, $\alpha=-2$ and $\beta=2\sqrt{2}$ in Proposition
\ref{prop:RootsExpression}, we have that for any $s\geq0$, the list
$\sigma_1=(\rho+s,\mu_-,\mu_+,0,\sigma)$ is realisable, where
\begin{equation}\label{eq:SpecialPhiFormula}
  \mu_\pm:=\frac{1}{2} \left(-4-s\pm\sqrt{16+8 s+s^2-\frac{288}{6+s}}\right).
\end{equation}
In particular, taking $s=2$, we have that $(8,-3,-3,0,\sigma)$ is realisable.

This example is reminiscent of the kind of perturbation given in Theorem \ref{thm:GuoGuo}, except
that we have also perturbed the imaginary part of the original complex conjugate pair
$-2\pm2\sqrt{2}i$. In fact, using a combination of Proposition \ref{prop:RootsExpression} and
Theorem \ref{thm:GuoGuo}, it is possible to show that
\begin{equation}\label{eq:SpecialList}
  (8,-3+ib,-3-ib,0,\sigma)
\end{equation}
is realisable for all $0\leq b\leq2\sqrt{2}$. To see this, let us label the expression under the
square root in (\ref{eq:SpecialPhiFormula}) as
\[
  h(s):=16+8 s+s^2-\frac{288}{6+s}.
\]
Since $h(0)=-32\leq -4b^2\leq 0=h(2)$ and $h$ is continuous on $[0,2]$, there
exists $s_0\in[0,2]$ such that $h(s_0)=-4b^2$. Then, taking $s=s_0$ gives the realisable list
\[
  \left(6+s_0,-2-\frac{s_0}{2}+ib,-2-\frac{s_0}{2}-ib,0,\sigma\right).
\]
Finally, letting $\delta=1-s_0/2$ in (\ref{eq:GuoGuoDecrease}), we may produce
(\ref{eq:SpecialList}).
\end{ex}

We finish this section with an example for which the limiting eigenvalues $\lambda_+$ and
$\lambda_-$ are complex:

\begin{ex}
  Let $\sigma$ be any list for which $\sigma_0=(2,i,-i,\sigma)$ is realisable. Applying
Theorem \ref{thm:p=3} with $a=0$, $\eta=1$ and $t=1$ produces the realisable spectrum
\[
  (2.4710\ldots,0.1868\ldots+(0.6666\ldots)i,0.1868\ldots-(0.6666\ldots)i,-0.8445\ldots,\sigma).
\]
$t=5$ gives
\[
  (3.8755\ldots,0.4100\ldots+(0.5573\ldots)i,0.4100\ldots-(0.5573\ldots)i,-2.6954\ldots,\sigma).
\]
$t=500$ gives
\[
  (32.1356\ldots,0.499\ldots+(0.5007\ldots)i,0.499\ldots-(0.5007\ldots)i,-31.1336\ldots,\sigma),
\]
illustrating the convergence of two of the eigenvalues of $\sigma_1$ to $\lambda_\pm=1/2\pm(1/2)i$.
\end{ex}

\bibliography{refs}
\bibliographystyle{plain}

\end{document}